\documentclass[final,3p,times]{elsarticle}

\usepackage{amsmath,amsthm,amscd,amsfonts,amssymb,epic,eepic,bbm,,graphicx,ytableau}
\usepackage{tikz}
\usepackage{tkz-euclide}
\usepackage{rotating}
\usetikzlibrary{angles,quotes}

\newtheorem{thm}{Theorem}[section]
\newtheorem{cor}[thm]{Corollary}

\newtheorem{lem}[thm]{Lemma}

\newtheorem{prop}[thm]{Proposition}

\newtheorem{example}[thm]{Example}
\theoremstyle{remark}
\newcommand{\la}{\lambda}

\begin{document}

\begin{frontmatter}



\title{A bijection between self-conjugate and ordinary partitions \\ and counting simultaneous cores as its application}


\author[1]{Hyunsoo Cho}
\ead{coconut@yonsei.ac.kr}
\author[2]{JiSun Huh\corref{cor1}}
\ead{hyunyjia@ajou.ac.kr}
\author[1]{Jaebum Sohn}
\ead{jsohn@yonsei.ac.kr}
\cortext[cor1]{Corresponding author}

\address[1]{Department of Mathematics, Yonsei University, Seoul 03722, Republic of Korea}
\address[2]{Department of Mathematics, Ajou University, Suwon 16499, Republic of Korea}
\begin{abstract}
We give a bijection between the set of self-conjugate partitions and that of ordinary partitions. Also, we show the relation between hook lengths of self conjugate partition and corresponding partition via the bijection. As a corollary, we give new combinatorial interpretations for the Catalan number and the Motzkin number in terms of self-conjugate simultaneous core partitions.

\end{abstract}

\begin{keyword}
partition, self-conjugate partition, hook length, simultaneous core partition



\end{keyword}

\end{frontmatter}


\section{Introduction}

Let $\la=(\la_1,\la_2,\dots,\la_{\ell})$ be a partition of $n$. The {\it Young diagram} of $\la$ is a collection of $n$ boxes in $\ell$ rows with $\la_i$ boxes in row $i$. We label the columns of the diagram from left to right starting with column 1. The box in row $i$ and column $j$ is said to be in position $(i,j)$. For example, the Young diagram for $\la=(5,4,2,1)$ is below.
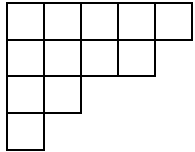
\begin{center}
\small{
\begin{ytableau}
~&~&~&~&~ \\
~&~&~&~ \\
~&~ \\
~
\end{ytableau}}
\end{center}
For the Young diagram of $\la$, the partition $\la '=(\la_1 ',\la_2 ',\dots, \la_{\la_1} ')$ is called the {\it conjugate} of $\la$, where $\la'_j$ denotes the number of boxes in column $j$. For each box in its Young diagram, we define its {\it hook length} by counting the number of boxes directly to its right or below, including the box itself. Equivalently, for the box in position $(i,j)$, the hook length of $\la$ is defined by 
$$h_\la(i,j)=\la_i+\la'_j-i-j+1.$$ 
For example, hook lengths in the first row above are $8$, $6$, $4$, $3$, and $1$, respectively. We denote $h_{\la}(i,j)$ by $h(i,j)$ when $\la$ is clear.

For a positive integer $t$, a partition $\la$ is called a {\it $t$-core} if none of its hook lengths are multiples of $t$. 
The number of $t$-core partitions of $n$ is denoted by $c_t(n)$. The study of core partitions arises from the representation theory of the symmetric group $S_n$. (See \cite{JK} for details)

In \cite{JK}, we can find the generating function of $c_t(n)$:

\begin{equation}\label{eqn:core}
\sum_{n=0}^{\infty}c_t(n)q^n=\prod_{n=1}^{\infty}\frac{(1-q^{nt})^t}{1-q^n}.
\end{equation}
Moreover, the generating function of $c_t(n)$ can be represented as the product of the Dedekind eta function. Therefore, many researches on core partitions are being made through various ways, such as representation theory and analytic methods--see, for example,  \cite{BN2,BY,GO,HS,Olsson,Olsson2}.
Stanton \cite{Stanton} conjectured the monotonicity of $c_t(n)$, that is, $c_{t+1}(n)\geq c_t(n)$ for $t\geq 4$ and $n\geq t+1$. Motivated by this conjecture, Anderson \cite{Anderson2}, Kim and Rouse \cite{KR} found the asymptotics of $c_t(n)$ and proved Stanton's conjecture partially. \\


A partition whose conjugate is equal to itself is called {\it self-conjugate}. Let $sc_t(n)$ denote the number of $t$-core partitions of $n$ which are self-conjugate. Garvan, Kim, and Stanton \cite{GKS} obtained the generating function of $sc_t(n)$. Among them, the generating function of $sc_{2t}(n)$ is  
 
\begin{equation}\label{eqn:selfcore}
\sum_{n=0}^{\infty}sc_{2t}(n)q^n=\prod_{n=1}^{\infty} (1-q^{4nt})^t (1+q^{2n-1}).
\end{equation} 

A number of properties of self-conjugate core partitions have been found and proved. (See, \cite{BDFKS,BN3}) Similarly, monotonicity conjecture on $sc_t(n)$ has been suggested in \cite{HN} and asymptotics are provided in \cite{Alpoge}. \\

The main result in this paper is to construct a bijection between the set $\mathcal{SC}^{(m)}$ of self-conjugate partitions with some restrictions and the set $\mathcal{P}$ of ordinary partitions which will be defined and explained in Sections \ref{sec:pre} and \ref{sec:main}.

\begin{thm}\label{thm:main}
There is a bijection between $\mathcal{SC}^{(m)}$ and $\mathcal{P}$.
\end{thm}

If we let $p(n)$ be the number of partitions of $n$ and $sc(n)$ be the number of self-conjugate partitions of $n$, then as a result of our main theorem we have the following corollary. 

\begin{cor}\label{cor:nocore}
For $|q|<1$,
$$\sum_{n=0}^{\infty}  sc(n) q^n =\left(\sum_{n=0}^{\infty} p(n) q^{4n}\right) \left(\sum_{n=0}^{\infty} q^{\frac{n(n+1)}{2}}\right).$$
\end{cor}

Moreover, from the bijection, we can find a relation between hook lengths of self-conjugate partition and that of corresponding partition. (See Theorem \ref{thm:second}) As an application of this relation, we can prove the following theorem. We use the notation of a $(t_1,...,t_p)$-core partition if it is simultaneously a $t_1$-core,\dots, and a $t_p$-core.

\begin{thm}\label{thm:simul}
 Let $c_{(t_1,\dots,t_p)}(n)$ be the number of $(t_1,\dots,t_p)$-core partitions of $n$ and $sc_{(t_1,\dots,t_p)}(n)$ be the number of self-conjugate $(t_1,\dots,t_p)$-core partitions of $n$. For $|q|<1$,
$$\sum_{n=0}^{\infty} sc_{(2t_1,...,2t_p)}(n) q^n 
=\left(\sum_{n=0}^{\infty}c_{(t_1,\dots,t_p)}(n)q^{4n}\right)\left(\sum_{n=0}^{\infty}q^{\frac{n(n+1)}{2}}\right).$$
\end{thm}

A special case of $p=1$ is the following corollary.

\begin{cor}\label{cor:core}
For $|q|<1$,
$$\sum_{n=0}^{\infty} sc_{2t}(n) q^n =\left(\sum_{n=0}^{\infty} c_t(n) q^{4n}\right) \left(\sum_{n=0}^{\infty} q^{\frac{n(n+1)}{2}}\right).$$
\end{cor}

We remark that Corollary \ref{cor:nocore} and Corollary \ref{cor:core} can also be proved by analytic method using equations (\ref{eqn:core}), (\ref{eqn:selfcore}), and Gauss's well-known identity

$$ \sum_{n=0}^{\infty}q^{\frac{n(n+1)}{2}}=\prod_{n=1}^{\infty}\frac{1-q^{2n}}{1-q^{2n-1}}.$$

\noindent However, it might be difficult to prove Theorem \ref{thm:simul} by analytic approach.

At the end of this paper, a new expression of the Catalan number and the Motzkin number in terms of self-conjugate simultaneous core partitions is given (See Corollary \ref{cor:catalanmotzkin}) as a corollary of Theorem \ref{thm:simul}.\\

This paper is organized as follows. In Section \ref{sec:pre}, we introduce new classification of the set of self-conjugate partitions. In Section \ref{sec:main}, we give a bijection between the set of self-conjugate partitions and that of ordinary partitions. In Section \ref{sec:hook}, we  explain the relation between even hook lengths in a self-conjugate partition and hook lengths in the corresponding partition via the bijection. Furthermore, we give some new results of the number of self-conjugate simultaneous core partitions.


\section{Preliminaries}\label{sec:pre}


We give some basic notions and introduce a set partition of the set $\mathcal{SC}$ of self-conjugate partitions. For a number $a$, $|a|$ is the absolute value, for the set $A$, $|A|$ is the cardinality of the set, and for a partition $\lambda$, $|\lambda|$ is the number of being partitioned.


\subsection{Main diagonal hook lengths}


Let $\la$ be a partition.  We often use the notation $\delta_i$ for the hook length $h(i,i)$ of the $i$th box on the main diagonal. The set $D(\la)=\{\delta_i~:~i=1,2,\dots\}$ is called the {\it set of main diagonal hook lengths of $\la$}. It is clear that if $\la$ is self-conjugate, then $D(\la)$ determines $\la$, and elements of $D(\la)$ are all distinct and odd. Hence, for a self-conjugate partition $\la$, $D(\la)$ can be divided into the following two subsets;
\begin{eqnarray*}
D_1(\la)&=&\{\delta_{i}\in D(\la)~:~\delta_{i}\equiv 1\pmod{4}\},\\
D_3(\la)&=&\{\delta_{i}\in D(\la)~:~\delta_{i}\equiv 3\pmod{4}\}.
\end{eqnarray*}

\begin{example}\label{ex:hook} 
Let $\la=(4,4,4,3)$ be a self-conjugate partition of $15$. Figure \ref{fig:example} shows the Young diagram and hook lengths of $\la$. The set $D(\la)=\{7,5,3\}$ of main diagonal hook lengths is divided into $D_1(\la)=\{5\}$ and $D_3(\la)=\{7,3\}$.

\begin{figure}[ht!]
\centering
\begin{ytableau}
{\bf 7}&6&5&3 \\ 6&{\bf 5}&4&2 \\ 5&4&{\bf 3}&1 \\ 3&2&1
\end{ytableau}
\caption{The Young diagram of a self-conjugate partition and its hook lengths}\label{fig:example}
\end{figure}
\end{example}

The set of hook lengths of boxes in the first column of the Young diagram of $\la$ is called the {\it beta-set} of $\la$ and denoted by $\beta(\la)$. If $\la'$ is the conjugate partition of $\la$, then the $\beta(\la')$ is the set of hook lengths of boxes in the first row of the Young diagram of $\la$. One may notice that if $\la$
is self-conjugate, then $\beta(\la)=\beta(\la')$. The following lemma plays an important role when we deal with hook lengths.

\begin{lem}\label{lem:hook} 
Let $\la=(\la_1,\la_2,\dots)$ be a self-conjugate partition with $D(\la)=\{\delta_1, \delta_2, \dots,\delta_d\}$. If we let
\begin{eqnarray*}
\beta_{\leq d}(\la)&=&\{h(i,1)~:~i=1,2,\dots,d\},\\
\beta_{> d}(\la)&=&\{h(i,1)~:~i=d+1,d+2,\dots,\la_1\},
\end{eqnarray*}
so that $\beta(\la)=\beta_{\leq d}(\la)\cup\beta_{> d}(\la)$, then
\begin{eqnarray*}
\beta_{\leq d}(\la)&=&\left\{\frac{\delta_1+\delta_1}{2},\frac{\delta_1+\delta_2}{2},\dots,\frac{\delta_1+\delta_d}{2}\right\},\\
\beta_{> d}(\la)&=&\left\{\frac{\delta_1-1}{2},\frac{\delta_1-3}{2},\dots,1\right\}-\left\{\frac{\delta_1-\delta_d}{2},\frac{\delta_1-\delta_{d-1}}{2},\dots,\frac{\delta_1-\delta_2}{2}\right\}.
\end{eqnarray*}
\end{lem}
\begin{proof}
We consider the set $\beta_{\leq d}(\la)$ first. Since $\la$ is self-conjugate, $\delta_i=2\la_i-2i+1$ for $i\leq d$. Hence, 
$$h(i,1)=\la_i+\la_1-i=\frac{\delta_1+\delta_i}{2} \qquad \text{for} \quad 1 \leq i \leq d.$$ 

Now, we consider the set $\beta_{>d}(\la)$.
If $\la$ has exactly $d$ parts, then one can notice that $\beta_{>d}(\la)$ must be an empty set. Since $\la$ is the partition of $d^2$ with $\la_i=d$ for all $i=1,2,\dots, d$, we have $\delta_d=1, ~\delta_{d-1}=3,~\dots, ~\delta_{2}=\delta_1-2$, and therefore,
$$\beta_{>d}(\la)=\left\{\frac{\delta_1-1}{2},\frac{\delta_1-3}{2}\dots,1\right\}-\left\{\frac{\delta_1-\delta_d}{2},\frac{\delta_1-\delta_{d-1}}{2},\dots,\frac{\delta_1-\delta_2}{2}\right\}=\emptyset.$$
If $\la$ has more than $d$ elements, then 
$$h(d+1,1)=\la_{d+1}+\la_{1}-(d+1)\leq d+\la_1-(d+1)=\la_1-1=\frac{\delta_1-1}{2}.$$ 
Since $h(d+1,1)$ is the largest element of $\beta_{>d}(\la)$, the set $\beta_{>d}(\la)$ is a subset of $\left\{\frac{\delta_1-1}{2},\frac{\delta_1-3}{2}\dots,1\right\}$. 
If we suppose that $\frac{\delta_1-\delta_j}{2}$ is an element of $\beta_{>d}(\la)$ for some $j\leq d$, then for some $i\in\{d+1,d+2,\dots,\la_1\}$, $$h(i,1)-\frac{\delta_1-\delta_j}{2}=(\la_1+\la_i-i)-(\la_1-\la_j+j-1)=\la_i+\la_j-i-j+1=0.$$
But $\la_i+\la_j-i-j+1$ is the hook length $h(i,j)$, so it should be nonzero even if the box $(i,j)$ is not in the Young diagram of $\la$. Hence, we have come to the conclusion that $\frac{\delta_1-\delta_j}{2}\notin \beta_{>d}(\la)$ for all $j\leq d$ and  
$$\beta_{>d}(\la)\subset\left\{\frac{\delta_1-1}{2},\frac{\delta_1-3}{2},\dots,1\right\}-\left\{\frac{\delta_1-\delta_d}{2},\frac{\delta_1-\delta_{d-1}}{2},\dots,\frac{\delta_1-\delta_2}{2}\right\}.$$
In fact, $\beta_{>d}(\la)=\left\{\frac{\delta_1-1}{2},\frac{\delta_1-3}{2},\dots,1\right\}-\left\{\frac{\delta_1-\delta_d}{2},\frac{\delta_1-\delta_{d-1}}{2},\dots,\frac{\delta_1-\delta_2}{2}\right\}$ since $|\beta_{>d}(\la)|=\la_1-d$.
\end{proof}  


\subsection{Self-conjugate partitions with same disparity}


We denote the set of self-conjugate partitions of $n$ by $\mathcal{SC}(n)$. Using the value $|D_1(\la)|-|D_3(\la)|$ of $\la\in\mathcal{SC}(n)$, we split $\mathcal{SC}(n)$ as follows: For nonnegative integers $m$ and $n$, we define a set $\mathcal{SC}^{(m)}(n)$ by
$$\mathcal{SC}^{(m)}(n)=\{\la\in\mathcal{SC}(n)~:~|D_1(\la)|-|D_3(\la)|=(-1)^{m+1} \left\lceil \frac{m}{2} \right\rceil\},$$ where $\lceil x \rceil$ is the least integer greater than or equal to $x$.
We note that for a self-conjugate partition $\la$, if $|D_1(\la)|-|D_3(\la)|=k$ for $k\geq1$, then $\la\in\mathcal{SC}^{(2k-1)}(n)$.
Otherwise, if $|D_1(\la)|-|D_3(\la)|=-k$ for $k\geq0$, then $\la\in\mathcal{SC}^{(2k)}(n)$.
Therefore, $\mathcal{SC}(n)=\bigcup\limits_{m=0}^{\infty}\mathcal{SC}^{(m)}(n)$. \\

We use the notation $sc^{(m)}(n)$ for $|\mathcal{SC}^{(m)}(n)|$. If we let $\mathcal{SC}^{(m)}=\bigcup_{n\geq 0} \mathcal{SC}^{(m)}(n)$, then  $\mathcal{SC}=\bigcup_{m\geq0} \mathcal{SC}^{(m)}$.\\ 

For a partition $\la$, we define the {\it disparity} of $\la$ by 
$${\rm dp}(\la)=|\{(i,j)\in\la~:~h(i,j) \text{~is odd}\}|-|\{(i,j)\in\la~:~h(i,j) \text{~is even}\}|.$$

For example, for the partition $\la=(4,4,4,3)$ given in Example \ref{ex:hook}, $|D_1(\la)|-|D_3(\la)|=-1$. Therefore, $\la$ is an element of $\mathcal{SC}^{(2)}(15)$ with ${\rm dp}(\la)=9-6=3$. \\

In the following proposition, we show that each element of the set $\mathcal{SC}^{(m)}(n)$ has the same disparity.

\begin{prop} \label{prop:oddeven}
For a nonnegative integer $m$, if $\la$ is in the set $\mathcal{SC}^{(m)}$, then its disparity ${\rm dp}(\la)$ is $\frac{m(m+1)}{2}$. 
\end{prop}

\begin{proof} Let $\la\in\mathcal{SC}^{(m)}$ be a self-conjugate partition with $D(\la)=\{\delta_1,\delta_2,\dots,\delta_d\}$. 

First, we focus on $d^2$ hook lengths $h(i,j)$ of $\la$ for $1\leq i,j \leq d$. Since $\delta_i=2\la_i-2i+1$ for $i\leq d$, 
$$h(i,j)=\la_i+\la_j-i-j+1=\frac{\delta_{i}+\delta_j}{2} \qquad \text{for} \quad 1\leq i,j \leq d.$$
Thus, among these $d^2$ hook lengths, there are $2|D_1(\la)||D_3(\la)|$ even numbers, and then 
$$|\{h(i,j) \text{: odd}~:~1\leq i,j\leq d\}|-|\{h(i,j) \text{: even}~:~1\leq i,j \leq d\}|=d^2-4|D_1(\la)||D_3(\la)|.$$ 

Now, we consider the multiset $H(\la)$ of hook lengths $h(i,j)$ of $\la$ for $i>d$ and $j\leq d$. If we let   
$$H_j(\la)=\{h(i,j)~:~d+1\leq i \leq \la_j\} \qquad \text{for} \quad 1\leq j\leq d,$$
then $H(\la)=\bigcup_{j=1}^{d} H_j(\la)$.
We note that $H_1(\la)$ is the set $\beta_{>d}(\la)$, $H_2(\la)$ is the set $\beta_{>{d-1}}(\bar{\la})$, where $\bar{\la}$ is the self-conjugate partition with $D(\bar{\la})=\{\delta_2,\delta_3,\dots,\delta_{d}\}$, and $H_3(\la),\dots,H_d(\la)$ are defined similarly. Hence, by Lemma \ref{lem:hook}, we have
\begin{eqnarray*}
H_1(\la)&=&\left\{\frac{\delta_1-1}{2},\frac{\delta_1-3}{2},\dots,1\right\}-\left\{\frac{\delta_1-\delta_d}{2},\frac{\delta_1-\delta_{d-1}}{2},\dots,\frac{\delta_{1}-\delta_{2}}{2}\right\},\\
H_2(\la)&=&\left\{\frac{\delta_2-1}{2},\frac{\delta_2-3}{2},\dots,1\right\}-\left\{\frac{\delta_2-\delta_d}{2},\frac{\delta_2-\delta_{d-1}}{2},\dots,\frac{\delta_{2}-\delta_{3}}{2}\right\},\\
& \vdots&\\
H_{d}(\la)&=&\left\{\frac{\delta_{d}-1}{2},\frac{\delta_{d}-3}{2},\dots,1\right\}.
\end{eqnarray*}
Thus, we can rewrite the multiset $H(\la)$ as follows.
$$H(\la)=\bigcup_{k=1}^{d}\left\{ \frac{\delta_{k}-1}{2}, \frac{\delta_{k}-3}{2}, \dots, 1 \right\}-\left\{ \frac{\delta_i - \delta_j}{2}~:~1\leq i<j\leq d\right\}.$$
We note that depend on whether $\delta_k \in D_1(\la)$ or $\delta_k \in D_3(\la)$, the number of odd elements is same as the number of even elements, or we have one more odd element in the set $\{ \frac{\delta_{k}-1}{2}, \frac{\delta_{k}-3}{2}, \dots, 1\}$.

Since there are $|D_1(\la)||D_3(\la)|$ odd elements in the multiset $\{ \frac{\delta_i - \delta_j}{2}~:~1\leq i<j\leq d\}$, we can conclude that among the hook lengths $h(i,j)$ with $i>d$ and $j\leq d$, 
\begin{eqnarray*}
|\{h(i,j):\text{odd}\}|-|\{h(i,j):\text{even}\}|&=&|D_3(\la)|-\left\{|D_1(\la)||D_3(\la)|-\left(\binom{d}{2}-|D_1(\la)||D_3(\la)|\right)\right\}.
\end{eqnarray*} 

Since $\la$ is self-conjugate, the multiset of hook lengths $h(i,j)$ with $i\leq d$ and $j>d$ is equal to the set $H(\la)$. Hence, the value $|\{h(i,j):\text{odd}\}|-|\{h(i,j):\text{even}\}|$ is equal to that one in the case when $i>d$ and $j\leq d$. By combining all of these, we have 
\begin{eqnarray*}
{\rm dp}(\la)&=&(d^2-4|D_1(\la)||D_3(\la)|)+2\left(|D_3(\la)|+\binom{d}{2}-2|D_1(\la)||D_3(\la)|\right) \\
&=&2d^2-d+2|D_3(\la)|-8|D_1(\la)||D_3(\la)|.
\end{eqnarray*}

Now, we consider two cases:

\begin{itemize}
\item Let $m=2k-1$ for a positive integer $k$. By the definition of $\mathcal{SC}^{(m)}$, $|D_1(\la)|-|D_3(\la)|=k$. Since $|D_1(\la)|+|D_3(\la)|=d$, we have $|D_1(\la)|=\frac{d+k}{2}$, $|D_3(\la)|=\frac{d-k}{2}$, and  
$${\rm dp}(\la)=2d^2-d+(d-k)-2(d^2-k^2)=2k^2-k=\frac{2k(2k-1)}{2}=\frac{(m+1)m}{2}.$$
\item Let $m=2k$ for a nonnegative integer $k$. In this case, $|D_1(\la)|-|D_3(\la)|=-k$. Hence, we have $|D_1(\la)|=\frac{d-k}{2}$, $|D_3(\la)|=\frac{d+k}{2}$, and
$${\rm dp}(\la)=2d^2-d+(d+k)-2(d^2-k^2)=2k^2+k=\frac{2k(2k+1)}{2}=\frac{m(m+1)}{2}.$$
\end{itemize}
This completes the proof.
\end{proof}

By Proposition \ref{prop:oddeven}, one may notice that the disparity of a self-conjugate partition is a triangular number $\frac{m(m+1)}{2}$ for some integer $m\geq 0 $, and the set of self-conjugate partitions with the disparity $\frac{m(m+1)}{2}$ is $\mathcal{SC}^{(m)}$.


\section{Proof of Theorem \ref{thm:main}} \label{sec:main}


The set of partitions of $n$ is denoted by $\mathcal{P}(n)$, and the set of partitions is denoted by $\mathcal{P}$. In this section, we construct a bijection between two sets $\mathcal{SC}^{(m)}(4n+m(m+1)/2)$ and $\mathcal{P}(n)$ which play a key role throughout the paper.\\ 

Before constructing a bijection, we give some notations.
For a self-conjugate partition $\la$, if
\begin{eqnarray*}
D_1(\la)&=&\{4a_1+1, 4a_2+1,\dots, 4a_r+1\},\\
D_3(\la)&=&\{4b_1-1, 4b_2-1,\dots, 4b_s-1\},
\end{eqnarray*}
we say that $\la$ has the {\it diagonal sequence pair} $((a_1,a_2,\dots,a_{r}),(b_1,b_2,\dots,b_{s}))$, where $a_1>a_2>\cdots>a_{r}\geq0$ and $b_1>b_2>\cdots>b_{s}\geq 1$. For convenience, we allow empty sequence if $r$ or $s$ is equal to $0$.

For $\la=(4,4,4,3)\in \mathcal{SC}^{(2)}(15)$ considered in Example \ref{ex:hook}, its diagonal sequence pair is $((1),(2,1))$. \\

We note that if $\la\in \mathcal{SC}^{(m)}(4n+m(m+1)/2)$ has the diagonal sequence pair $((a_1,\dots,a_{r}),(b_1,\dots,b_s))$, then 
$$r-s+(-1)^m \left\lceil \frac{m}{2} \right\rceil =0$$ 
and 
$$4\left(\sum_{i=1}^{r}a_i +\sum_{j=1}^{s}b_j\right)+r-s=4n+\frac{m(m+1)}{2}.$$
 
Now, we are ready to construct our mapping.\\

\textbf{Mapping $\phi^{(m)}_n$ : $\mathcal{SC}^{(m)}(4n+m(m+1)/2)~\rightarrow~\mathcal{P}(n)$}\\

Let $\la\in\mathcal{SC}^{(m)}(4n+m(m+1)/2)$ with the diagonal sequence pair $((a_1,\dots,a_r),(b_1,\dots,b_s))$. We define $\phi^{(m)}_n(\la)$ by the partition $\mu=(\mu_1, \mu_2,\dots, \mu_{\ell})$ such that 
$$\mu_i=a_{i}+i+s-r \qquad \text{for} \quad i \leq r,$$ 

\noindent and $(\mu_{r+1},\mu_{r+2}, \dots, \mu_{\ell})$ is the conjugate of the partition $\gamma=(b_1-s, b_2-s+1, \dots, b_{s}-1)$.

\noindent (We allow that $\gamma$ has some zero parts.) \\

In the figure below, the diagram after deleting the gray portion is the Young diagram of $\mu$.

\begin{figure}[ht]
\centering
\begin{tikzpicture}[scale=.43]
\filldraw[fill=gray!30,color=gray!30] (0,0) rectangle (3,7);

\foreach \i in {1,2,3,4,5,6}
 \draw[dashed] (\i,0) -- (\i,7-\i); 

\draw[thick] (0,0) -- (0,7)
                    (0,0) -- (7.03,0)
                    (1,7) -- (15,7) -- (15,6) --(2,6)
                    (14,6) -- (14,5) --(3,5)
                    (12,5) -- (12,4) --(4,4)
                    (12,4) -- (12,3) --(5,3)
                    (11,3) -- (11,2) --(6,2)
                    (10,2) -- (10,1) --(7,1)
                    (10,1) -- (10,0) --(6,0)
                    (3,0) -- (3,-6) -- (4,-6) --(4,0)
                    (4,-5) -- (5,-5) -- (5,0)
                    (5,-5) -- (6,-5) -- (6,0)
                    (6,-4) -- (7,-4) -- (7,0);

\draw[<->] (-0.5,0) -- (-0.5,4);
\draw[<->] (-0.5,4) -- (-0.5,7);
\draw[<->] (0,7.5) -- (3,7.5);
\draw[<->] (3,7.5) -- (7,7.5);

\node at (-1,2) {\begin{turn}{90} 
$s$
\end{turn}};
\node at (-1,5.5) {\begin{turn}{90} 
$r-s$
\end{turn}};
\node at (1.5, 8) {$r-s$};
\node at (5, 8) {$s$};
\node at (8.5, 0.45) {$a_r$};
\node at (8.5, 1.45) {$a_{r-1}$};
\node at (8.5, 2.45) {$a_{r-2}$};
\node at (8.5, 3.7) {$\vdots$};
\node at (8.5, 4.45) {$a_{3}$};
\node at (8.5, 5.45) {$a_{2}$};
\node at (8.5, 6.45) {$a_{1}$};

\node at (3.5, -2) {\begin{turn}{90} 
$b_1-s$
\end{turn}};
\node at (4.5, -2) {\begin{turn}{90} 
$b_2-s+1$
\end{turn}};
\node at (5.55,-2) {$\dots$};
\node at (6.5, -2) {\begin{turn}{90} 
$b_s-1$
\end{turn}};

\foreach \i in {0,1,2,3,4,5,6}
 \draw[thick] (0,\i+1) -- (7-\i,\i+1) --  (7-\i,\i); 

\node at (7.5,-7) {i) $r > s$}; 
\end{tikzpicture}
\quad\qquad\begin{tikzpicture}[scale=.43]

\foreach \i in {1,2,3}
 \draw[dashed] (\i,0) -- (\i,3); 

\draw[thick] (0,0) -- (0,3) -- (14,3) -- (14,2) -- (0,2); 
\draw[thick] (12,2) -- (12,1) -- (0,1); 
\draw[thick] (10,1) -- (10,0) -- (0,0); 
\draw[thick] (0,0) -- (0,-10) -- (1,-10) --(1,0); 
\draw[thick] (1,-8) -- (2,-8) -- (2,0); 
\draw[thick] (2,-7) -- (3,-7) -- (3,0); 
\draw[thick] (3,-7) -- (4,-7) -- (4,3);
\draw[thick] (4,-5) -- (5,-5) -- (5,0);  
\draw[thick] (5,-4) -- (6,-4) -- (6,0);
\draw[thick] (6,-4) -- (7,-4) -- (7,1);    
\draw[thick] (6,1) -- (6,2);    
\draw[thick] (5,2) -- (5,3);    
\draw[dashed] (5,2) -- (5,0);  
\draw[dashed] (6,1) -- (6,0);  

\node at (8.5,0.45) {$a_r$};
\node at (8.5,1.7) {$\vdots$};
\node at (8.5,2.45) {$a_1$};

\node at (0.5, -2) {\begin{turn}{90} 
$b_1-s$
\end{turn}};
\node at (1.5, -2) {\begin{turn}{90} 
$b_2-s+1$
\end{turn}};
\node at (2.5, -2) {\begin{turn}{90} 
$b_3-s+2$
\end{turn}};
\node at (3.55,-2) {$\dots$};
\node at (4.5, -2) {\begin{turn}{90} 
$b_{s-2}-3$
\end{turn}};
\node at (5.5, -2) {\begin{turn}{90} 
$b_{s-1}-2$
\end{turn}};
\node at (6.5, -2) {\begin{turn}{90} 
$b_s-1$
\end{turn}};

\draw[<->] (-0.5,0) -- (-0.5,3);
\draw[<->] (0,3.5) -- (4,3.5);
\draw[<->] (4,3.5) -- (7,3.5);

\node at (2, 4) {$s-r$};
\node at (5.5, 4) {$r$};
\node at (-1,1.5) {\begin{turn}{90} 
$r$
\end{turn}};

\node at (7,-11) {ii) $r \leq s$}; 
\end{tikzpicture}
\caption{Graphical interpretations of mapping $\phi^{(m)}_n$}
\end{figure}
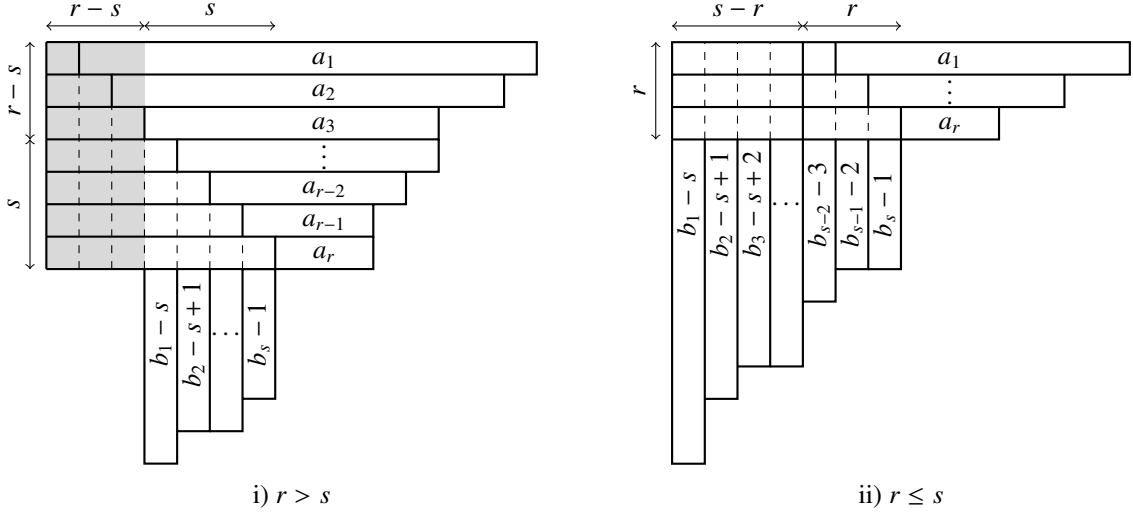

Now, we show that the mapping $\phi_n^{(m)}$ is well-defined:\\
Since $\mu_i=a_i+i+s-r$ for $i\leq r$, $\mu_r=a_r+r+s-r=a_r+s$. From $a_1>a_2>\cdots>a_r\geq 0$, we have $\mu_1\geq \mu_2\geq \cdots \geq \mu_r \geq s$. Since $\gamma $ has at most $s$ parts, $\mu_{r+1}\leq s$. Therefore, $\mu$ is a partition of
$$|\mu|=\sum_{i=1}^{r}(a_i+i+s-r) +\sum_{j=1}^{s}(b_j-s+j-1).$$

\begin{itemize}
\item  If $m=2k-1$, then $r-s=k$ and 
$$|\mu|=\left(\sum_{i=1}^{r}a_i +\sum_{j=1}^{s}b_{j}\right)-\frac{k(k-1)}{2}
=\frac{4n+2k^2-2k}{4}-\frac{k(k-1)}{2}=n.$$
\item If $m=2k$, then $r-s=-k$ and 
$$|\mu|=\left(\sum_{i=1}^{r}a_i +\sum_{j=1}^{s}b_{j}\right)-\frac{k(k+1)}{2}
=\frac{4n+2k^2+2k}{4}-\frac{k(k+1)}{2}=n.$$
\end{itemize}

\begin{thm}\label{thm:bijective}
For nonnegative integers $m$ and $n$, the mapping $\phi_n^{(m)}$ is bijective.
\end{thm}
\begin{proof}
We prove the theorem by constructing the inverse mapping $\psi_n^{(m)}$. 
Let $\mu=(\mu_1,\mu_2,\dots,\mu_{\ell})\in \mathcal{P}(n)$. 
\begin{itemize}
\item For $m=2k-1$, there exist a unique $s\geq0$ such that $\mu_{s+k}\geq s$ and $\mu_{s+k+1}\leq s$. Let $r=s+k$.
\item For $m=2k$, there exist a unique $r\geq0$ such that $\mu_{r}\geq r+k$ and $\mu_{r+1}\leq r+k$. Let $s=r+k$.
\end{itemize}
For any $m$, let $\gamma=(\gamma_1,\gamma_2,\dots,\gamma_s)$ be the conjugate of the partition $(\mu_{r+1},\dots,\mu_{\ell})$. Now, we define $\psi_n^{(m)}(\mu)$ by the partition $\la$ with the diagonal sequence pair $((a_1,\dots,a_{r}),(b_1,\dots,b_s))$, where
\begin{eqnarray*}
a_{i}&=&\mu_i-i-s+r \qquad \text{for} \quad i \leq r,\\
b_{j}&=&\gamma_j+s-j+1 \qquad \text{for} \quad j \leq s.
\end{eqnarray*}
Then $\psi_n^{(m)}$ is well-defined. Moreover, one can see that $\phi_n^{(m)}(\psi_n^{(m)}(\mu))=\mu$ and $\psi_n^{(m)}(\phi_n^{(m)}(\la))=\la$. Therefore, $\psi_n^{(m)}$ is the inverse of the mapping $\phi_n^{(m)}$.
\end{proof}

We define the bijection $\phi^{(m)}:\mathcal{SC}^{(m)} \rightarrow \mathcal{P}$ by $\phi^{(m)}_n(\la)$, for a partition $\la\in\mathcal{SC}^{(m)}$ of  $4n+\frac{m(m+1)}{2}$.\\
We say that $\mu$ is the {\it corresponding partition} of $\la$ when $\phi^{(m)}(\la)=\mu$. \\ 

We give two examples of the bijection $\phi^{(m)}$.

\begin{example} \label{ex:bijection}
We consider two self-conjugate partitions $\la$ and $\tilde{\la}$ with the set of main diagonal hook lengths $D(\la)=\{21,15,13,9,3,1\}$ and $D(\tilde{\la})=\{31,19,11,5\}$, respectively. 

\begin{itemize}
\item Since $D_1(\la)=\{21,13,9,1\}$ and $D_3(\la)=\{15,3\}$, $\la$ is in the set $\mathcal{SC}^{(3)}$ and $((5,3,2,0),(4,1))$ is the diagonal sequence pair of $\la$. If we let $\mu$ be the partition $\phi_{14}^{(3)}(\la)$, then 
$$\mu_1=5+1-2=4,\quad \mu_2=3+2-2=3, \quad \mu_3=2+3-2=3, \quad \mu_4=0+4-2=2$$ 
and $(\mu_{5},\mu_6,\dots)$ is the conjugate of the partition $(4-2,1-2+1)$. Therefore, $\mu=(4,3,3,2,1,1)$.
\item Since $D_1(\tilde{\la})=\{5\}$ and $D_3(\tilde{\la})=\{31,19,11\}$, $\tilde{\la} \in \mathcal{SC}^{(4)}$ and $((1),(8,5,3))$ is the diagonal sequence pair of $\tilde{\la}$. If we let $\tilde{\mu}$ be the partition $\phi_{14}^{(4)}(\tilde{\la})$, then $\mu_1=1+1+2=4$ and 
$(\mu_{2},\mu_3,\dots)$ is the conjugate of the partition $(8-3,5-3+1,3-3+2)$. Therefore, $\tilde{\mu}=(4,3,3,2,1,1)$.
\end{itemize}

For given $\mu\in\mathcal{P}$ and $m\geq0$, we consider the following diagram to find the correspondence of $\mu$. For convenience, even if $i\leq 0$, we set the $i$th column is the column on the left side of the $(i+1)$st column and the $i$th row is on the above of the $(i+1)$st row.
\begin{itemize}
\item For $m=2k-1$, we consider the diagram $\nu$ obtained from the Young diagram of $\mu$ by attaching $\frac{k(k-1)}{2}$ boxes on the left side such that $\nu$ has $\mu_i+k-i$ boxes in row $i$ for $i<k$ and $\mu_i$ boxes in row $i$ for $i\geq k$. 
Then, the number of (white) boxes $(i,j)$ in row $i$ such that $i-j<k$ is equal to $a_i$ and the number of (gray) boxes $(i,j)$ in column $j$ such that $i-j\geq k$ is equal to $b_j$. See the first diagram in Figure \ref{fig:graphic} for $\mu=(4,3,3,2,1,1)$  and $m=3$.  
\item For $m=2k$, we consider the diagram $\nu$ obtained from the Young diagram of $\mu$ by attaching $\frac{k(k+1)}{2}$ boxes on the above such that $\nu$ has $k-i$ boxes in row $-i$ for $i=0,1,\dots,k-1$ and $\mu_i$ boxes in row $i$ for $i>0$. 
Then, the number of (white) boxes $(i,j)$ in row $i$ such that $i-j<-k$ is equal to $a_i$ and the number of (gray) boxes $(i,j)$ in column $j$ such that $i-j\geq -k$ is equal to $b_j$. See the second diagram in Figure \ref{fig:graphic} for $\mu=(4,3,3,2,1,1)$  and $m=4$.   
\end{itemize}

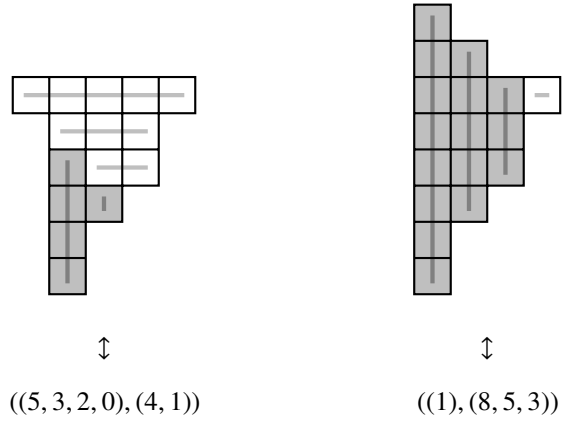
\begin{figure}[ht]
\centering
\begin{tikzpicture}[scale=.48]
\filldraw[fill=gray!50,color=gray!50] (0,-2) rectangle (1,-6)
                                                   (1,-3) rectangle (2,-4)
                                                   (10,2) rectangle (11,-6)
                                                   (11,1) rectangle (12,-4)
                                                   (12,0) rectangle (13,-3);

\draw[ultra thick,color=gray!50] (-0.7,-0.5) -- (3.7,-0.5);
\draw[ultra thick,color=gray!50] (0.3,-1.5) -- (2.7,-1.5);
\draw[ultra thick,color=gray!50] (1.3,-2.5) -- (2.7,-2.5);

\draw[ultra thick,color=gray] (0.5,-2.3) -- (0.5,-5.7);
\draw[ultra thick,color=gray] (1.5,-3.3) -- (1.5,-3.7);

\draw[thick] (0,0) -- (0,-6) --(1,-6) -- (1,0)
                    (0,-4) -- (2,-4) -- (2,0) 
                    (0,-3) -- (3,-3) -- (3,0)
                    (0,-1) -- (4,-1) -- (4,0) -- (0,0)
                    (0,-2) -- (3,-2)
                    (0,-5) -- (1,-5)
                    (0,0) -- (-1,0) -- (-1,-1) -- (0,-1);
 
\draw[ultra thick,color=gray] (10.5,1.7) -- (10.5,-5.7);
\draw[ultra thick,color=gray] (11.5,0.7) -- (11.5,-3.7);
\draw[ultra thick,color=gray] (12.5,-0.3) -- (12.5,-2.7);

\draw[ultra thick,color=gray!50] (13.3,-0.5) -- (13.7,-0.5);

\draw[thick] (10,0) -- (10,-6) --(11,-6) -- (11,0)
                    (10,-4) -- (12,-4) -- (12,0) 
                    (10,-3) -- (13,-3) -- (13,0)
                    (10,-1) -- (14,-1) -- (14,0) -- (10,0)
                    (10,-2) -- (13,-2)
                    (10,-5) -- (11,-5)
                    (10,0) -- (10,2) -- (11,2) -- (11,0)
                    (10,1) -- (12,1) -- (12,0);
               
\node at (1.5,-7.5) {$\updownarrow$};
\node at (12,-7.5) {$\updownarrow$};
\node at (1.5,-9) {$((5,3,2,0),(4,1))$};
\node at (12,-9) {$((1),(8,5,3))$};
\end{tikzpicture}
\caption{Graphical interpretations for odd $m$ and even $m$ of the bijection $\phi^{(m)}$}\label{fig:graphic}
\end{figure}

\end{example}

\begin{prop}\label{prop:size}    
For integer $m\geq 0$, the number of self-conjugate partitions of $n$ with the disparity $\frac{m(m+1)}{2}$ is
$$sc^{(m)}(n)=\begin{cases}
               p(k) & \text{if}~n=4k+\frac{m(m+1)}{2},\\
               0    & \text{otherwise}.
               \end{cases}$$
\end{prop}

\begin{proof}
For a self-conjugate partition $\la$ with the disparity $\frac{m(m+1)}{2}$, suppose there are $\ell$ boxes $(i,j)$ in the Young diagram of $\la$ such that $h(i,j)$ is even. By Proposition \ref{prop:oddeven}, there are $\ell+\frac{m(m+1)}{2}$ boxes with odd hook lengths. Since $\la$ is self-conjugate, $\ell$ must be even. If we let $\ell=2k$, then $n=4k+\frac{m(m+1)}{2}$. Bijection $\phi_n^{(m)}$ gives that $sc^{(m)}\left(4k+\frac{m(m+1)}{2}\right)=p(k)$.
\end{proof}

By Theorem \ref{thm:bijective} and Proposition \ref{prop:size}, we have the following corollary and as a consequence of Corollary \ref{cor:first}, we have Corollary \ref{cor:nocore}.

\begin{cor}\label{cor:first}
For a nonnegative integer $m$, we have 
$$\sum_{\la\in\mathcal{SC}^{(m)}} q^{|\la|}=q^{\frac{m(m+1)}{2}}\sum_{\mu\in\mathcal{P}}q^{4|\mu|}.$$
\end{cor}

  
\section{Relations between hook lengths of $\mathcal{SC}^{(m)}$ and $\mathcal{P}$}\label{sec:hook}


In this section we provide some relations between hook lengths of $\la\in \mathcal{SC}^{(m)}$ and that of $\phi^{(m)}(\la)\in \mathcal{P}$. 


\subsection{Hook lengths of the first row or column}        
                        

For the partitions $\la\in \mathcal{SC}^{(m)}$ and $\mu=\phi^{(m)}(\la)$, we give a relation between their hook lengths in the first row or in the first column.

\begin{lem} \label{lem:firsthook}
For a partition $\mu=(\mu_1,\mu_2,\dots,\mu_{\ell})$, if we let $\bar{\beta}(\mu)=\{h(1,1)-h(i,1)~:~2\leq i \leq \ell\}$, then 
$$\bar{\beta}(\mu)=\{1,2,\dots,h(1,1)\}-\beta(\mu'),$$
where $\mu'$ is the conjugate of $\mu$.
\end{lem}

\begin{proof}
For some $i$ and $j$, if $h(1,1)-h(i,1)=h(1,j)$, then $\mu_i+\mu_j'-i-j+1=0$. Since any hook length $h(i,j)$ can not be zero, $h(1,1)-h(i,1)\neq h(1,j)$ for all $i$ and $j$. This implies that $\bar{\beta}(\mu)\subset \{1,2,\dots,h(1,1)\}-\beta(\mu')$.
Since $|\bar{\beta}(\mu)|=h(1,1)-\mu_1$, we are done.
\end{proof}
        
For a self-conjugate partition $\la$, we define the {\it half-even beta set} of $\la$ by
$$\beta_{e/2}(\la)=\{h(i,1)/2~:~h(i,1)~\text{is even},~1\leq i \leq \la_1\}.$$             
                
\begin{prop} \label{prop:evenhook}  
Let $\la\in\mathcal{SC}^{(m)}$ be a partition with $D(\la)=\{\delta_1,\delta_2,\dots,\delta_d\}$. If $\mu=(\mu_1,\mu_2,\dots,\mu_{\ell})$ is the corresponding partition of $\la$, then the half-even beta set of $\la$ is
$$\beta_{e/2}(\la)=\begin{cases}
                   \beta(\mu') &\text{if}\quad \delta_1\in D_1(\la),\\
                   \beta(\mu)  &\text{if}\quad \delta_1\in D_3(\la).
                   \end{cases}$$  
\end{prop}            

\begin{proof}
We let $((a_1,\dots,a_{r}),(b_1,\dots,b_s))$ be the diagonal sequence pair of $\la$ so that
$$D(\la)=\{4a_1+1,4a_2+1,\dots,4a_{r}+1\}\cup\{4b_1-1,4b_2-1,\dots,4b_s-1\}.$$

First we will find two sets $\beta(\mu)$ and $\beta(\mu')$. The partition $\phi^{(m)}(\la)=\mu$ is defined by 
$$\mu_i=a_{i}+i+s-r \qquad \text{for} \quad i \leq r,$$ 
and $(\mu_{r+1}, \dots, \mu_{\ell})$ is the conjugate of the partition $\gamma=(b_1-s, b_2-s+1, \dots, b_s-1)$. We suppose that $\gamma$ has $u$ nonzero parts so that
$b_{i}-s+i-1=0$ for $i>u$. 

To find $\beta(\mu)$, we consider hook lengths of boxes $(i,1)$ in the Young diagram of $\mu$. For $i\leq r$,
$$h_{\mu}(i,1)=\mu_i+\mu_1'-i=(a_i+i+s-r)+(r+b_1-s)-i=a_i+b_1.$$
We note that the set $\{h_{\mu}(i,1)~:~i>r\}$ is equal to the set $\beta(\gamma')$, where $\gamma'$ is the conjugate of $\gamma$. Since
$$h_{\gamma}(i,1)=\gamma_i+\gamma'_1-i=(b_i-s+i-1)+u-i=b_i-s+u-1,$$
we have
$$\bar{\beta}(\gamma)=\{h_{\gamma}(1,1)-h_{\gamma}(i,1)~:~2\leq i \leq u\}
                   =\{b_1-b_i~:~2\leq i \leq u \}.$$
By Lemma \ref{lem:firsthook},
$$\beta(\gamma')=\{1,2,\dots,h_{\gamma}(1,1)\}-\bar{\beta}(\gamma)
                   =\{1,2,\dots,b_1-s+u-1\}-\{b_1-b_i~:~2\leq i \leq u \}.$$
Since $\{b_1-b_i~:~u<i\leq s\}=\{b_1-s+i-1~:~u<i\leq s\}$, we can write as 
$$\beta(\gamma')=\{1,2,\dots,b_1-1\}-\{b_1-b_i~:~2\leq i \leq s \},$$  
and therefore, the beta set of $\mu$ is 
$$\beta(\mu)=\{a_i+b_1~:~1\leq i \leq r\}\cup\{1,2,\dots,b_1-1\}-\{b_1-b_i~:~2\leq i \leq s\}.$$

To find the set $\beta(\mu')$, we consider hook lengths of boxes $(1,j)$ in the Young diagram of $\mu$.\\  
For $j\leq s$, since $\mu'_j=r+\gamma_j$, 
$$h_{\mu}(1,j)=\mu_1+\mu_j'-j=(a_1+1+s-r)+(r+b_j-s+j-1)-j=a_1+b_j.$$

We note that if we set $\nu=(\mu_1-s, \mu_2-s, \dots, \mu_r-s)$, then the set $\{h_{\mu}(1,j)~:~j>s\}$ is equal to the set $\beta(\nu')$, where $\nu'$ is the conjugate of $\nu$. We suppose that $\nu$ has $v$ nonzero parts so that $a_i+i-r=0$ for $i>v$. Since 
$$h_{\nu}(i,1)=\nu_i+\nu'_1-i=(\mu_i-s)+v-i=(a_i+i+s-r-s)+v-i=a_i-r+v,$$ 
we have
$$\bar{\beta}(\nu)=\{h_{\nu}(1,1)-h_{\nu}(i,1)~:~2\leq i \leq v\}
=\{a_1-a_i~:~2\leq i \leq v\}.$$
By Lemma \ref{lem:firsthook},
$$\beta(\nu')=\{1,2,\dots,h_{\nu}(1,1)\}-\bar{\beta}(\nu)
             =\{1,2,\dots,a_1-r+v\}-\{a_1-a_i~:~2\leq i \leq v \}.$$  
Since $\{a_1-a_i~:~v<i\leq r\}=\{a_1-r+i~:~v<i\leq r\}$, we can write as 
$$\beta(\nu')=\{1,2,\dots,a_1\}-\{a_1-a_i~:~2\leq i \leq r \},$$  
and therefore, we have the beta set of $\mu'$,
$$\beta(\mu')=\{a_1+b_j~:~1\leq j \leq s\}\cup\{1,2,\dots,a_1\}-\{a_1-a_i~:~2\leq i \leq r \}.$$

\noindent Now, we consider the half-even beta set $\beta_{e/2}(\la)$. By Lemma \ref{lem:hook}, we have  
$$\beta(\la)=\left\{\frac{\delta_1+\delta_1}{2},\frac{\delta_1+\delta_2}{2},\dots,\frac{\delta_1+\delta_d}{2}\right\}
\bigcup \left\{\frac{\delta_1-1}{2},\frac{\delta_1-3}{2},\dots,1\right\}-\left\{\frac{\delta_1-\delta_d}{2},\frac{\delta_1-\delta_{d-1}}{2},\dots,\frac{\delta_1-\delta_2}{2}\right\}.$$

We consider two cases:
\begin{enumerate}
\item[]\textbf{Case 1)} Let $\delta_1\in D_1(\la)$ so that $\delta_1=4a_1+1$. Then, 
$$\beta_{e/2}(\la)=\{a_1+b_j~:~1\leq j \leq s\}\cup\{1,2,\dots,a_1 \}-\{a_1-a_i~:~2\leq i \leq r\}=\beta(\mu').$$
\item[]\textbf{Case 2)} Let $\delta_1\in D_3(\la)$ so that $\delta_1=4b_1-1$. Then, 
$$\beta_{e/2}(\la)=\{b_1+a_i~:~1\leq i\leq r\}\cup\{1,2,\dots,b_1-1\}-\{b_1-b_i~:~2\leq i \leq s\}=\beta(\mu).$$
\end{enumerate}
\end{proof}        
  
\begin{example} Let $\la$, $\tilde{\la}$ be self-conjugate partitions we considered in Example \ref{ex:bijection}. We remind that $\phi^{(3)}(\la)=\phi^{(4)}(\tilde{\la})=\mu=(4,3,3,2,1,1)$. We note that $h_{\la}(1,1)=21\in D_{1}(\la)$ and $h_{\tilde{\la}}(1,1)=31\in D_{3}(\tilde{\la})$. As in Proposition \ref{prop:evenhook}, $\beta_{e/2}(\la)=\beta(\mu')=\{9,6,4,1\}$ and $\beta_{e/2}(\tilde{\la})=\beta(\mu)=\{9,7,6,4,2,1\}$. See Figure \ref{fig:hookrelation} for the Young diagrams of $\mu$, $\la$, $\tilde{\la}$, and their hook lengths. 
\end{example}
\begin{figure}[ht!]
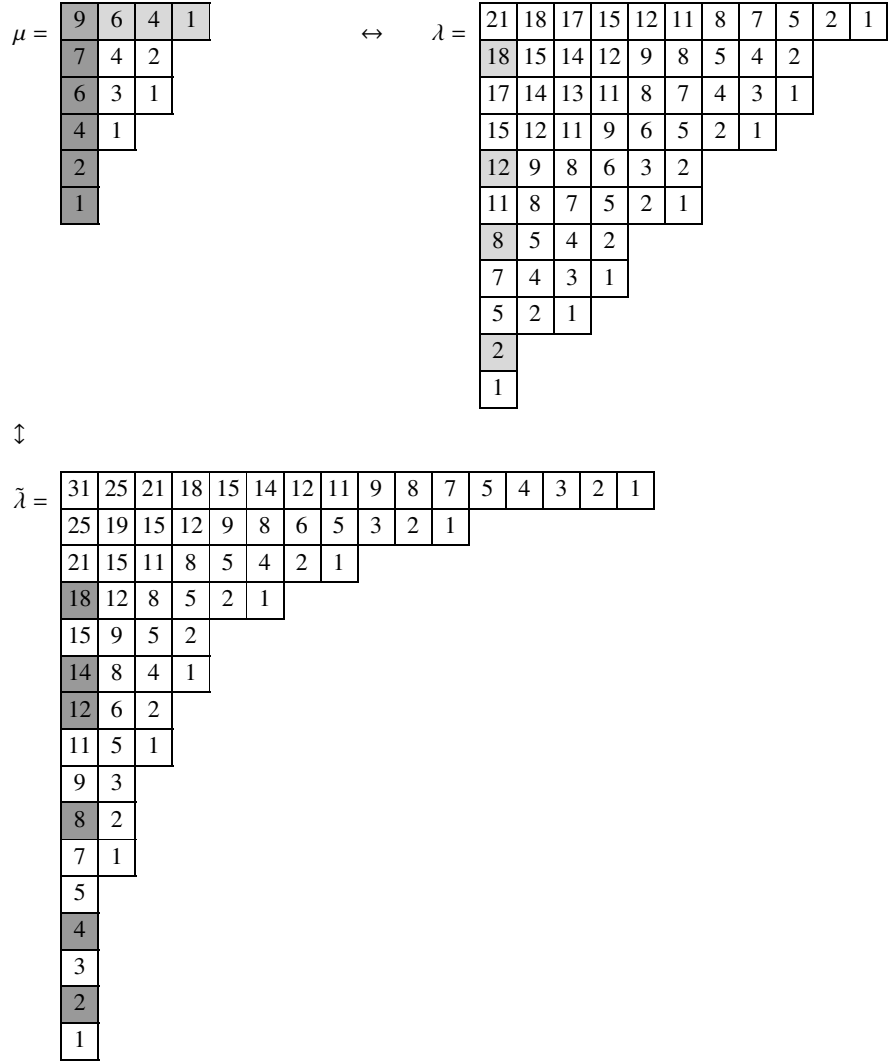

\small
\qquad \qquad \qquad
$\qquad\mu=~$\begin{ytableau}      
*(gray!65)9&*(gray!30)6&*(gray!30)4&*(gray!30)1\cr
*(gray!80)7&4&2\cr
*(gray!80)6&3&1\cr
*(gray!80)4&1\cr
*(gray!80)2\cr
*(gray!80)1\cr
\end{ytableau}
\qquad\qquad\qquad $\leftrightarrow$\qquad $\la=~$\begin{ytableau}
21&18&17&15&12&11&8&7&5&2&1\cr
*(gray!30)18&15&14&12&9&8&5&4&2\cr
17&14&13&11&8&7&4&3&1\cr
15&12&11&9&6&5&2&1\cr
*(gray!30)12&9&8&6&3&2\cr
11&8&7&5&2&1\cr
*(gray!30)8&5&4&2\cr
7&4&3&1\cr
5&2&1\cr
*(gray!30)2\cr
1\cr
\end{ytableau}\\

$\qquad\qquad\qquad\qquad\updownarrow$\\

$\qquad\qquad\qquad\qquad\tilde{\la}=~$\begin{ytableau}
31&25&21&18&15&14&12&11&9&8&7&5&4&3&2&1\cr
25&19&15&12&9&8&6&5&3&2&1\cr
21&15&11&8&5&4&2&1\cr
*(gray!80)18&12&8&5&2&1\cr
15&9&5&2\cr
*(gray!80)14&8&4&1\cr
*(gray!80)12&6&2\cr
11&5&1\cr
9&3\cr
*(gray!80)8&2\cr
7&1\cr
5\cr
*(gray!80)4\cr
3\cr
*(gray!80)2\cr
1\cr
\end{ytableau}
\caption{Hook length relations between corresponding partitions}\label{fig:hookrelation}
\end{figure}
  

\subsection{Relations between hook lengths of $\mathcal{SC}^{(m)}$ and $\mathcal{P}$}        
                        
  
One may notice that there are more relations between hook lengths of corresponding partitions from Figure \ref{fig:hookrelation}. Now, we give our main theorem.
 
\begin{thm}\label{thm:second} Let $\la\in \mathcal{SC}^{(m)}$ be a self-conjugate partition and let $\mu$ be its corresponding partition. For each positive integer $k$, the number of boxes $(i,j)$ with $h_{\la}(i,j)=2k$ is equal to twice the number of boxes $(\tilde{i},\tilde{j})$ with $h_{\mu}(\tilde{i},\tilde{j})=k$.  
\end{thm}

The following proposition is necessary to prove Theorem \ref{thm:second}.
     
\begin{prop} \label{prop:delete} 
For $\la\in\mathcal{SC}^{(m)}$ with $D(\la)=\{\delta_1,\delta_2,\dots,\delta_d\}$, let $\bar{\la}$ be the self-conjugate partition with $D(\bar{\la})=\{\delta_i\in D(\la)~:~2\leq i \leq d\}$, and let $\mu=(\mu_1,\mu_2,\dots,\mu_{\ell})$ and $\bar{\mu}$ be the corresponding partitions of $\la$ and $\bar{\la}$, respectively. If $\mu=(\mu_1,\mu_2,\dots,\mu_{\ell})$, then 
$$\bar{\mu}=\begin{cases}
            (\mu_2,\mu_3,\dots,\mu_{\ell}) & \text{if}\quad \delta_1\in D_1(\la)\\  (\mu_1-1,\mu_2-1,\dots,\mu_{\ell}-1) & \text{if}\quad \delta_1\in D_3(\la)
                \end{cases} $$ 
\end{prop}

\begin{proof}
We let $((a_1,\dots,a_{r}),(b_1,\dots,b_s))$ be the diagonal sequence pair of $\la$ so that $\phi^{(m)}(\la)=\mu$ is constructed by 
$$\mu_i=a_{i}+i+s-r \qquad \text{for} \quad i \leq r,$$ 
and $(\mu_{r+1}, \dots, \mu_{\ell})$ is the conjugate of the partition $\gamma=(b_1-s, b_2-s+1, \dots, b_s-1)$.\\

We denote $\bar{\mu}=(\bar{\mu}_1,\bar{\mu}_2,\dots)$ and consider the following two cases.

\begin{enumerate}
\item[]\textbf{Case 1)} If $\delta_1\in D_1(\la)$, the diagonal sequence pair of $\bar{\la}$ is $((a_2,\dots,a_{r}),(b_1,\dots,b_s))$ for $r\geq2$ or $(\emptyset,(b_1,\dots,b_s))$ for $r=1$. Then the corresponding partition $\bar{\mu}$ of $\bar{\la}$ is defined by 
$$\bar{\mu}_i=a_{i+1}+i+s-(r-1)=a_{i+1}+(i+1)+s-r\qquad \text{for} \quad i \leq r-1,$$ 
and $(\bar{\mu}_{r},\bar{\mu}_{r+1}, \dots)$ is the conjugate of $(b_1-s, b_2-s+1, \dots, b_s-1)$. Hence, $\bar{\mu}_i=\mu_{i+1}$ for all $i$.
\item[]\textbf{Case 2)} If $\delta_1\in D_3(\la)$, the diagonal sequence pair of $\bar{\la}$ is $((a_1,\dots,a_{r}),(b_2,\dots,b_s))$ for $s\geq2$ or $((a_1,\dots,a_r),\emptyset)$ for $s=1$. In this case, the corresponding partition $\bar{\mu}$ of $\bar{\la}$ is defined by 
$$\bar{\mu}_i=a_{i}+i+(s-1)-r=a_{i}+i+s-r-1\qquad \text{for} \quad i \leq r,$$ 
and $(\bar{\mu}_{r+1},\bar{\mu}_{r+2}, \dots)$ is the conjugate of $\bar{\gamma}=(b_2-(s-1), \dots, b_s-1)=(b_2-s+1,\dots,b_s-1)$. Therefore, $\bar{\mu}_i=\mu_{i}-1$ for all $i$.
\end{enumerate}
\end{proof}              

\begin{proof}[Proof of Theorem \ref{thm:second}]
We apply strong induction on $|\la|$. For $|\la|=0$, $\la=\emptyset\in\mathcal{SC}^{(0)}$. Since $\phi^{(0)}(\la)=\emptyset$, the assertion holds.

Now, we assume that the assertion holds for $|\la|<n$. Let $\la$ be a self-conjugate partition of $n\geq1$ with $D(\la)=\{\delta_1,\delta_2,\dots,\delta_d\}$ and let $\mu=(\mu_1,\mu_2,\dots,\mu_{\ell})$ be the corresponding partition of $\la$.
If $\bar{\la}$ is the self-conjugate partition of $n-\delta_1$ with $D(\bar{\la})=\{\delta_i~:~2\leq i \leq d\}$ and $\bar{\mu}$ is its corresponding partition, then by our assumption, the number of boxes $(i,j)$ with $h_{\bar{\la}}(i,j)=2k$ is equal to twice the number of boxes $(\tilde{i},\tilde{j})$ with $h_{\bar{\mu}}(\tilde{i},\tilde{j})=k$ for each $k\geq 1$. We consider the following two cases.
\begin{enumerate}
\item[]\textbf{Case 1)} Let $\delta_1\in D_1(\la)$. By Proposition \ref{prop:delete}, $\bar{\mu}=(\mu_2,\mu_3,\dots,\mu_{\ell})$. Therefore, it is enough to show that 
the number of boxes $(i,j)$ either in the first row or in the first column of $\la$ with $h_{\la}(i,j)=2k$ is equal to twice the number of boxes $(\tilde{i},\tilde{j})$ in the first row in $\mu$ with $h_{\mu}(\tilde{i},\tilde{j})=k$ for each $k\geq 1$.
By Proposition \ref{prop:evenhook}, $\beta_{e/2}(\la)=\beta(\mu')$. Moreover, $\beta_{e/2}(\la')=\beta(\mu')$ since $\la$ is self-conjugate.

\item[]\textbf{Case 2)} Let $\delta_1\in D_3(\la)$. By Proposition \ref{prop:delete}, $\bar{\mu}=(\mu_1-1,\mu_2-1,\dots,\mu_{\ell}-1)$. Therefore, it is enough to show that the number of boxes $(i,j)$ either in the first row or in the first column of $\la$ with $h_{\la}(i,j)=2k$ is equal to twice the number of boxes $(\tilde{i},\tilde{j})$ in the first column in $\mu$ with $h_{\mu}(\tilde{i},\tilde{j})=k$ for each $k\geq 1$.
By Proposition \ref{prop:evenhook}, $\beta_{e/2}(\la)=\beta_{e/2}(\la')=\beta(\mu)$, as we desired.
\end{enumerate}
\end{proof}    

The following corollary is obtained directly from Theorem \ref{thm:second}.

\begin{cor}
For a self-conjugate partition $\lambda\in\mathcal{SC}^{(m)}$, let $\mu$ be the corresponding partition of $\lambda$. Then $\lambda$ is a $(2t_1,2t_2,\dots,2t_p)$-core partition if and only if $\mu$ is a $(t_1,t_2,\dots,t_p)$-core partition.
\end{cor}

We denote the set of self-conjugate $(t_1,t_2,\dots,t_p)$-core partitions $\la\in\mathcal{SC}^{(m)}$ of $n$ by $\mathcal{SC}^{(m)}_{(t_1,...,t_p)}(n)$, and use notation $sc^{(m)}_{(t_1,...,t_p)}(n)$ for 
$|\mathcal{SC}^{(m)}_{(t_1,...,t_p)}(n)|$.\\

By using Theorem \ref{thm:bijective} and Theorem \ref{thm:second}, we obtain the cardinality of $\mathcal{SC}^{(m)}_{(2t_1,...,2t_p)}(n)$.

\begin{prop}\label{prop:coresize} For $m\geq0$, the number of self-conjugate $(2t_1,2t_2,\dots,2t_p)$-core partitions of $n$ with the disparity $\frac{m(m+1)}{2}$ is
$$sc_{(2t_1,...,2t_p)}^{(m)}(n)=\begin{cases}
                c_{(t_1,...,t_p)}(k) & {\text if} \quad n=4k+\frac{m(m+1)}{2},\\
                               0     & {\text otherwise}.
                               \end{cases}$$
\end{prop}


From the previous proposition, we have Theorem \ref{thm:simul}.



\subsection{Counting self-conjugate $(2t_1,\dots,2t_p)$-core partitions with the disparity $\frac{m(m+1)}{2}$}        
                        

In this subsection, we give some sets of self-conjugate partitions each of them is counted by known special numbers.
It is well-known that there are finitely many $(t_1,...,t_p)$-core partitions when  $t_1,...,t_p$ are relatively prime numbers. From Proposition \ref{prop:coresize}, we have the following result.  

\begin{cor}\label{cor:sim}
For relatively prime numbers $t_1,...,t_p$, the number of self-conjugate $(2t_1,...,2t_p)$-core partitions with the disparity $\frac{m(m+1)}{2}$ is equal to the number of $(t_1,...,t_p)$-core partitions.
\end{cor}

Anderson \cite{Anderson} gives an expression for the Catalan number in terms of simultaneous core partitions, and Amderberhan and Leven \cite{AL}, Yang, Zhong, and Zhou \cite{YZZ}, Wang \cite{Wang}, respectively, gives an identity for the Motzkin number. 

\begin{thm}\cite{Anderson}\label{thm:anderson}
For relatively prime integers $t_1,t_2\geq1$, the number of $(t_1,t_2)$-core partitions is
$$c_{(t_1,t_2)}=\frac{1}{t_1+t_2}\binom{t_1+t_2}{t_1}.$$ 
In particular, $c_{(n,n+1)}=C_n$, where $C_n=\frac{1}{n+1}\binom{2n}{n}$ is the $n$th Catalan number.
\end{thm}

\begin{thm}\cite{Wang}\label{thm:Wang}
For relatively prime integers $n,d\geq1$, the number of $(n,n+d,n+2d)$-core partitions is
$$c_{(n,n+d,n+2d)}=\frac{1}{n+d}\sum_{i=0}^{\lfloor n/2\rfloor}\binom{n+d}{i,i+d,n-2i}.$$
In particular, $c_{(n,n+1,n+2)}$ is the $n$th Motzkin number $M_n=\sum_{i\geq 0} \frac{1}{i+1}\binom{n}{2i}\binom{2i}{i}$.
\end{thm}  
    
By using Corollary \ref{cor:sim} and the above known results, we have the following corollary.

\begin{cor}\label{cor:catalanmotzkin}
Let $m\geq0$ be an integer.
\begin{enumerate} 
\item[\rm{(a)}] For relatively prime integers $t_1,t_2\geq1$, the number of self-conjugate $(2t_1,2t_2)$-core partitions with the disparity $\frac{m(m+1)}{2}$ is
$$sc^{(m)}_{(t_1,t_2)}=\frac{1}{t_1+t_2}\binom{t_1+t_2}{t_1}.$$
In particular, $sc^{(m)}_{(2n,2n+2)}=C_n$, where $C_n$ is the $n$th Catalan number.
\item[\rm{(b)}] For relatively prime integers $n,d\geq1$, the number of self-conjugate $(2n,2n+2d,2n+4d)$-core partitions with the disparity $\frac{m(m+1)}{2}$ is 
$$sc^{(m)}_{(2n,2n+2d,2n+4d)}=\frac{1}{n+d}\sum_{i=0}^{\lfloor n/2\rfloor}\binom{n+d}{i,i+d,n-2i}.$$
In particular, $sc^{(m)}_{(2n,2n+2,2n+4)}$ is the $n$th Motzkin number.
\end{enumerate}
\end{cor}

\section*{Acknowledgement}
The third named author's work was supported by the National Research Foundation of Korea (NRF) NRF-2017R1A2B4009501.



\bibliographystyle{elsarticle-harv} 
\bibliography{mybib}





\end{document}